\documentclass[paper=a4,english,fontsize=11pt,parskip=half,abstract=true]{scrartcl}
\usepackage{babel}
\usepackage[utf8]{inputenc}
\usepackage[T1]{fontenc}
\usepackage[left=20mm,right=20mm,top=30mm,bottom=30mm]{geometry}
\usepackage{amsmath}
\usepackage{amsthm}
\usepackage{amssymb}
\usepackage{enumerate} 
\usepackage{thmtools}
\usepackage{mathtools}
\usepackage{booktabs}
\mathtoolsset{centercolon} 
\usepackage[bookmarks=true,
            pdftitle={On redundant Sylow subgroups},
            pdfauthor={Benjamin Sambale},
            pdfkeywords={},
            pdfstartview={FitH}]{hyperref}

\newtheorem{Thm}{Theorem} 
\newtheorem{Lem}[Thm]{Lemma}

\theoremstyle{definition}

\numberwithin{equation}{section}

\allowdisplaybreaks[1]
\newenvironment{sm}{\bigl(\begin{smallmatrix}}{\end{smallmatrix}\bigr)}

\renewcommand{\phi}{\varphi}
\newcommand{\C}{\mathrm{C}}
\newcommand{\N}{\mathrm{N}}

\newcommand{\pcore}{\mathrm{O}}

\newcommand{\NN}{\mathbb{N}}
\newcommand{\FF}{\mathbb{F}}

\newcommand{\GL}{\operatorname{GL}}
\newcommand{\SL}{\operatorname{SL}}

\newcommand{\PSL}{\operatorname{PSL}}

\newcommand{\Syl}{\operatorname{Syl}}

\title{On redundant Sylow subgroups}
\author{Benjamin Sambale\footnote{Institut für Algebra, Zahlentheorie und Diskrete Mathematik, Leibniz Universität Hannover, Welfengarten 1, 30167 Hannover, Germany,
\href{mailto:sambale@math.uni-hannover.de}{sambale@math.uni-hannover.de}}}
\date{\today}

\begin{document}
\frenchspacing
\maketitle

\begin{abstract}\noindent
A Sylow $p$-subgroup $P$ of a finite group $G$ is called \emph{redundant} if every $p$-element of $G$ lies in a Sylow subgroup different from $P$. Generalizing a recent theorem of Maróti--Martínez--Moretó, we show that for every non-cyclic $p$-group $P$ there exists a solvable group $G$ such that $P$ is redundant in $G$. Moreover, we answer several open questions raised by Maróti--Martínez--Moretó. 
\end{abstract}


\section{Introduction}

For a prime $p$, let $\nu_p(G)$ be the number of Sylow $p$-subgroups of a finite group $G$. Mikko Korhonen~\cite{Korhonen} has asked 10 years ago whether there exist groups $G$ such that the set of $p$-elements $G_p$ of $G$ can be covered by less than $\nu_p(G)$ Sylow $p$-subgroups, i.\,e. $G_p$ is the union of those Sylow subgroups. A positive answer was given by Jack Schmidt~\cite{Korhonen} using a group $G$ with elementary abelian Sylow $p$-subgroups. 
Most recently, Maróti--Martínez--Moretó~\cite{MMM} have called a Sylow $p$-subgroup $P$ of $G$ \emph{redundant} if $G_p$ is already covered by the Sylow subgroups different from $P$ (by Sylow's theorem, either all or none Sylow subgroups are redundant). Obviously, redundant Sylow subgroups must be non-cyclic. Generalizing Schmidt's construction, they showed in \cite[Theorem~A]{MMM} (using a deep theorem of Turull and the solvable case of Thompson's theorem) that for every non-cyclic $p$-group $P$ of exponent $p$ there exists a solvable group $G$ such that $P$ is redundant in $G$. 
They have further speculated on p.~483 that the restriction on the exponent of $P$ might be superfluous. 
In this paper, we show in \autoref{main} that this is indeed the case. Our construction is in fact much easier and does not depend on the theorems of Turull and Thompson mentioned above. Using a refined method in \autoref{refine}, we also provide examples where $\nu_p(G)$ only depends on $p$. In particular, we show that $\nu_2(G)=27$ is the smallest possible value for a group $G$ with a redundant Sylow $2$-subgroup. We also the determine the smallest value of $\nu_p(G)$ for $p$-solvable groups in \autoref{psolv}.
This is a contribution to \cite[Question~8.5]{MMM}.

In \cite[p.~846]{MMM}, the authors state that there are very few groups $G$ such that $|G_p|<\nu_p(G)$ and only examples with elementary abelian Sylow $p$-subgroups are known. Our construction yields such examples for every non-cyclic $p$-group $P$. This leads to a negative answer to \cite[Question~8.7]{MMM}. 
On the other hand, we provide a positive answer to \cite[Question~8.8]{MMM} in \autoref{Q88}.

\section{Results}

\begin{Thm}\label{main}
For every non-cyclic $p$-group $P$ and every prime $q\ne p$ there exists an elementary abelian $q$-group $N$ such that $P$ acts on $N$ and $G:=N\rtimes P$ has the following properties:
\begin{enumerate}[(i)]
\item\label{a} $P$ is redundant in $G$.
\item\label{b} $|G_p|\le\frac{1}{q^{p-1}}|G|$.
\item\label{c} $G_p$ is covered by $\frac{1}{q^{p-1}}\nu_p(G)$ Sylow $p$-subgroups.
\end{enumerate}
\end{Thm}
\begin{proof}\hfill
\begin{enumerate}[(i)]
\item Let $V$ be the regular $\FF_qP$-module with basis $B:=\{v_x:x\in P\}$. Then $P$ acts trivial on $Z:=\langle\prod_{x\in P}v_x\rangle$. Let $N:=V/Z\cong C_q^{|P|-1}$ and $G:=N\rtimes P$. Since $P$ acts transitively on $B$, it follows that $\C_N(P)=1$ and $\N_G(P)=P$. Let $x\in P$.
Since $P$ is not cyclic, 
\[w:=\prod_{c\in\langle x\rangle}v_cZ\in\C_N(x)\setminus\{1\}.\] 
Hence, $x=wxw^{-1}\in wPw^{-1}\in\Syl_p(G)\setminus\{P\}$. This shows that $P$ is redundant in $G$.

\item Let $R\subseteq P$ be a set of representatives for the conjugacy classes of $P$. By construction, every $p$-element of $G$ is conjugate to a unique element $x\in R$. Let $g\in\C_G(x)$ and write $g=ny$ with $n\in N$ and $y\in P$. Then $xy\equiv xg\equiv gx\equiv yx\pmod{N}$ and therefore $[x,y]\in P\cap N=1$. This shows that $y\in\C_P(x)$, $n\in\C_N(x)$ and $\C_G(x)=\C_N(x)\C_P(x)$.
Every right coset $C$ of $\langle x\rangle$ in $P$ determines an element $w_C:=\prod_{c\in C}v_c\in\C_V(x)$. It is easy to check that the elements $\{w_C:C\in\langle x\rangle\backslash P\}$ form a basis of $\C_V(x)$. This yields 
\[|\C_N(x)|=|\C_V(x)/Z|= q^{|P:\langle x\rangle|-1}\ge q^{p-1}.\]
Hence,
\[|G_p|=\sum_{x\in R}|G:\C_G(x)|=\sum_{x\in R}|P:\C_P(x)||N:\C_N(x)|\le \frac{|N|}{q^{p-1}}\sum_{x\in R}|P:\C_P(x)|=\frac{1}{q^{p-1}}|G|.\]

\item Since $P$ is non-cyclic, there exist maximal subgroups $P_1,\ldots,P_{p+1}\le P$ such that $P=P_1\cup\ldots\cup P_{p+1}$. 
Then $N_i:=\C_N(P_i)\cong C_q^{p-1}$ for $i=1,\ldots,p+1$ by the argument of \eqref{b}. Since $P_j\unlhd P$, each $P_i$ acts on $N_j$. For $i\ne j$, we have $N_i\cap N_j=\C_N(\langle P_i,P_j\rangle)=\C_N(P)=1$. By the Fitting decomposition (see \cite[Theorem~4.34]{IsaacsGroup}), we obtain 
\[N_j=\C_{N_j}(P_i)\times[P_i,N_j]=[P_i,N_j]\le[P_i,N].\] 
Since $N=N_i\times [P_i,N]$, it follows that 
\[N_i\cap \prod_{j\ne i}N_j\le N_i\cap[P_i,N]=1.\] 
Therefore, $N_1\times\ldots\times N_{p+1}\le N$. We choose a basis $b_{i,1},\ldots,b_{i,p-1}$ of $N_i$ for every $i=1,\ldots,p+1$. Then the elements $b_{i,j}$ are linearly independent and can be extended to a basis $B$ of $N$. For $w\in N$ and $b\in B$ let $w_b$ be the coefficient of $w$ with respect to $b$. Define
\[T:=\Bigl\{w\in N:\forall j=1,\ldots,p-1:\sum_{i=1}^{p+1}w_{b_{i,j}}\equiv 0\pmod{q}\Bigr\}.\]
Then $|T|=\frac{1}{q^{p-1}}|N|$. Let $n\in N$ and $x\in P$ be arbitrary. There exist $i$ and $t\in T$ such that $x\in P_i$ and $t_b=n_b$ for all $b\in B\setminus\{b_{i,1},\ldots,b_{i,p-1}\}$. It follows that $t^{-1}n\in N_i\le\C_N(x)$ and $nxn^{-1}=txt^{-1}\in tPt^{-1}$. Hence, $G_p$ is covered by $\{tPt^{-1}:t\in T\}$. 
\qedhere
\end{enumerate}
\end{proof}

If $q^{p-1}>|P|$ in the situation of \autoref{main}, then $|G_p|<|N|=|G:\N_G(P)|=\nu_p(G)$ by \eqref{b}. 
If $p$ or $q$ goes to infinity, \eqref{c} furnishes a counterexample to \cite[Question~8.7]{MMM}. At the same time, it provides some evidence for \cite[Question~8.6]{MMM}. If $P$ contains a cyclic subgroup of index $p$, one can show that $G_p$ cannot be covered by less than $\frac{1}{q^{p-1}}\nu_p(G)$ Sylow subgroups.

If $P$ is the Klein four-group and $q=3$, then the construction of the proof above yields the group $G\cong\mathtt{SmallGroup}(108,40)$ with $\nu_2(G)=27$, which was mentioned in \cite[Introduction]{MMM}.
Question~8.5 of \cite{MMM} asks for the smallest possible value of $\nu_p(G)$ when $G$ has a redundant Sylow $p$-subgroup. Our proof of \autoref{main} yields $\nu_p(G)=|N|=q^{|P|-1}$. We give a better bound, which only depends on $p$. 

\begin{Thm}\label{refine}
For every non-cyclic $p$-group $P$ there exists a solvable group $G$ such that $P$ is redundant in $G$ and $\nu_p(G)=q^{p+1}$, where $q>1$ is the smallest prime power congruent to $1$ modulo $p$.
\end{Thm}
\begin{proof}
Since $P$ is non-cyclic, there exist maximal subgroups $P_1,\ldots,P_{p+1}\le P$ such that $P=P_1\cup\ldots\cup P_{p+1}$. Since $q\equiv 1\pmod{p}$, the finite field $\FF_q$ contains a primitive $p$-th root of unity. Hence, for $i=1,\ldots,p+1$ there exists a $1$-dimensional $\FF_qP$-module $N_i$ with kernel $P_i$. Define $N=N_1\oplus\ldots\oplus N_{p+1}$. Since every $x\in P$ lies in some $P_i$, it follows that $\C_N(x)>1=\C_N(P)$. Now by the proof of \autoref{main}\eqref{a} (or using \cite[Corollary~3.2]{MMM}), it follows that $P$ is redundant in $G:=N\rtimes P$ and $\nu_p(G)=|N|=q^{p+1}$ (we do not need that $P$ acts faithfully on $N$). 
\end{proof}

\autoref{refine} provides the following upper bounds for the minimal values of $\nu_p(G)$: 
\[
\begin{array}{c|cccccccccc}
p&2&3&5&7&11&13&17&19&23&29\\\toprule
\min\limits_{q\,\equiv\, 1\,(\text{mod }p)}q^{p+1}&3^3&2^8&11^6&2^{24}&23^{12}&3^{42}&103^{18}&191^{20}&47^{24}&59^{30}
\end{array}
\]

Now we work in the opposite direction by finding lower bounds on $\nu_p(G)$. 
We following result settles the case $p=2$.

\begin{Thm}\label{p2}
Let $G$ be a finite group with a redundant Sylow $2$-subgroup. Then $\nu_2(G)\ge 27$.
\end{Thm}
\begin{proof}
Let $N$ be the kernel of the conjugation action of $G$ on $\Syl_2(G)$, i.\,e. $N$ is the intersection of all Sylow normalizers.
Let $P\in\Syl_2(G)$. Since $P$ is the unique Sylow $2$-subgroup of $PN$, the map $\Syl_2(G)\to\Syl_2(G/N)$, $P\mapsto PN/N$ is a bijection and $P$ is redundant in $\Syl_2(G)$ if and only if $PN/N$ is redundant in $G/N$. Hence, we may assume that $N=1$. Then $G$ is a transitive permutation group of degree $\nu_2(G)$. We run through the database of all transitive groups of odd degree up to $25$ in GAP~\cite{GAPnew}. For each such group we can quickly check whether the stabilizer has a normal Sylow $2$-subgroup. If this is the case, we check whether $G$ has a redundant Sylow $2$-subgroup using \cite[Lemmas~2.1 and 2.6]{MMM}. It turns out that there are no examples with $\nu_2(G)<27$. 
\end{proof}

With the same method, we obtain $\nu_3(G)\ge 49$ and $\nu_5(G)\ge 51$ whenever $G$ has a redundant Sylow $p$-subgroup for $p=3$ or $p=5$ respectively.
The next lemma improves \cite[Theorem~8.4]{MMM} with an easier proof.

\begin{Lem}\label{lem}
Let $G$ be a finite group with a redundant Sylow $p$-subgroup. Then $\nu_p(G)\ge p^2+p+1$.
\end{Lem}
\begin{proof}
Let $P\in\Syl_p(G)$ be covered by $P_1,\ldots,P_k\in\Syl_p(G)\setminus\{P\}$ such that $k$ is as small as possible. 
Then $P\cap P_i\ne P\cap P_j$ for $i\ne j$. Since $P$ is not the union of $p$ proper subgroups, we must have $k\ge p+1$. 
Let $g\in\N_P(P\cap P_i)\setminus P_i$. Then $g\notin\N_G(P_i)$, since otherwise $P_i\langle g\rangle$ would be a $p$-subgroup larger than $P_i$. Hence, the Sylow subgroups $P_i^{g^j}$ for $j=1,\ldots,p$ are pairwise distinct and 
\[P\cap P_i^{g^j}=P^{g^j}\cap P_i^{g^j}=(P\cap P_i)^{g^j}=P\cap P_i.\]
In this way we obtain $kp$ Sylow $p$-subgroups different from $P$. Hence, $\nu_p(G)\ge kp+1\ge p^2+p+1$.
\end{proof}

\begin{Lem}
Let $G$ be a finite group with a redundant Sylow $p$-subgroup. Then $\nu_p(G)$ is not a prime.
\end{Lem}
\begin{proof}
Let $G$ be a minimal counterexample with $P\in\Syl_p(G)$ redundant. 
As in the proof of \autoref{p2}, we may assume that $G$ is a transitive permutation group of prime degree $q:=|\Syl_p(G)|$. 
By a result of Burnside, $G$ is a subgroup of the affine group $C_q\rtimes C_{q-1}$ or a $2$-transitive almost simple group (see \cite[Corollary~3.5B and Theorem~4.1B]{DM}). The first case is impossible, since $P$ must be non-cyclic. The latter case can be investigated with the classification of the finite simple groups (see \cite[p. 99]{DM} or \cite{Guralnickpprim}).
More precisely, the socle $N$ of $G$ is one of the following simple groups:
\begin{enumerate}[(i)]
\item $N=A_q$. Since the stabilizer $A_{q-1}$ must have a normal Sylow $p$-subgroup, it follows that $q=5$ and $p=2$. 
By \autoref{p2}, neither $G=A_5$ nor $G=S_5$ has a redundant Sylow $2$-subgroup. 

\item $N=\PSL(2,11)$ with $q=11$. Here $|G:N|\le 2$ and the stabilizer is isomorphic to $A_5$, so it cannot have a normal Sylow $p$-subgroup.

\item $N=M_{11}=G$ with $q=11$. Again the stabilizer $M_{10}$ has no normal Sylow $p$-subgroup.

\item $N=M_{23}=G$ with $q=23$. Here the stabilizer $M_{22}$ is simple.

\item $N=\PSL(n,r)$ with $q=\frac{r^n-1}{r-1}$ where $n$ is a prime. Suppose that $n\mid r-1$. Then $q=1+r+\ldots+r^{n-1}\equiv n\equiv 0\pmod{n}$ and $q=n$. But this contradicts $q>r-1$. Hence, $\gcd(n,r-1)=1$ and $N=\SL(n,r)$. 
Here $G$ acts on the set of lines or hyperplanes of $\FF_r^n$. In both cases the stabilizer, say $N_v$ contains a copy of $\GL(n-1,r)$. If $n>2$, then $|\GL(n-1,r)|$ is divisible by $r\frac{r^{n-1}-1}{r-1}=q-1$. In particular, $N_v$ has a non-trivial Sylow $p$-subgroup, which cannot be normal since $\GL(n-1,r)$ is involved in $N_v$. Consequently, $n=2$ and $q=r+1$ is a Fermat prime. Now $G/N$ is a cyclic $2$-group. For $p>2$ it is well-known that the Sylow $p$-subgroup of $N$ and $G$ are cyclic (see \cite[8.6.9]{Kurzweil}). Hence, $p=2$ and $G=PN$. The upper unitriangular matrices constitute a Sylow $2$-subgroup $Q\le P$ of $N$. We consider $x:=\begin{sm}
1&1\\0&1
\end{sm}\in Q$. It is easy to see that $\C_N(x)=Q$. In particular, $Q$ is the only Sylow $2$-subgroup of $N$ containing $x$. Since $\N_G(P)\le\N_G(P\cap N)=\N_G(Q)$ and $\nu_p(G)=q$ is a prime, we have 
\[\nu_p(N)=|N:\N_N(Q)|=|N:N\cap\N_G(Q)|=|N\N_G(Q):\N_G(Q)|\bigm||G:\N_G(P)|=\nu_p(G)\]
and $\nu_p(N)=\nu_p(G)$.
Therefore, $P$ is the only Sylow $2$-subgroup of $G$ containing $Q$ and $x$. Thus, $P$ is not redundant and we derived a contradiction. \qedhere
\end{enumerate}
\end{proof} 

Now we consider $p$-solvable groups.
For $H\le P\in\Syl_p(G)$ let $\lambda_G(H)$ be the number of Sylow $p$-subgroups of $G$ containing $H$. The following result was proved using Wielandt's subnormalizers.

\begin{Lem}[Casolo]\label{casolo}
Let $G$ be a $p$-solvable group and $H\le P\in\Syl_p(G)$. Let $\mathcal{M}$ be the set of $p'$-quotients in a normal series of $G$ whose quotients are $p$-groups or $p'$-groups. Then
\[\lambda_G(H)|\N_G(P):P|=\prod_{Q\in\mathcal{M}}|\C_Q(H)|.\]
\end{Lem}
\begin{proof}
See Theorems~2.6 and 2.8 in \cite{Casolo}.
\end{proof}

\begin{Thm}\label{psolv}
Let $G$ be a $p$-solvable group with a redundant Sylow $p$-subgroup. Then $\nu_p(G)\ge q^{p+1}$, where $q>1$ is the smallest prime power congruent to $1$ modulo $p$.
\end{Thm}
\begin{proof}
Let $P\in\Syl_p(G)$ and $\mathcal{M}$ as in \autoref{casolo}. Choosing $H=P$ in \autoref{casolo} yields
\[|\N_G(P):P|=\prod_{Q\in\mathcal{M}}|\C_Q(P)|.\]
Let $N:=\bigtimes_{Q\in\mathcal{M}}Q$ and $\tilde{G}:=N\rtimes P$. Then $\nu_p(G)=|G:\N_G(P)|=|\tilde{G}:\N_{\tilde{G}}(P)|=\nu_p(\tilde{G})$. Now let $H\le P$ be a cyclic subgroup. Since $P$ is redundant in $G$, we have $\lambda_G(H)>1$. In this situation \autoref{casolo} shows that $\lambda_{\tilde{G}}(H)>1$. Hence, $P$ is redundant in $\tilde{G}$ and we may assume that $G=\tilde{G}$ is $p$-nilpotent and $N=\pcore_{p'}(G)$.
Then $\C_N(x)>\C_N(P)$ for all $x\in P$. We consider $N$ as a $P$-set via the conjugation action.
By a theorem of Hartley--Turull~\cite{HartleyTurull} (see also \cite[Theorem~3.31]{IsaacsGroup}), there exists an abelian group $A$ and an isomorphism of $P$-sets $\phi\colon N\to A$, i.\,e. $\phi(n^x)=\phi(n)^x$ for all $x\in P$ and $n\in N$. 
In particular, $\C_A(x)=\phi(\C_N(x))>\phi(\C_N(P))=\C_A(P)$. 
Hence, $P$ is redundant in $A\rtimes P$ and 
\[\nu_p(A\rtimes P)=|A:\C_A(P)|=|N:\C_N(P)|=\nu_p(G).\]
Thus, we may assume that $N=A$ is abelian. Then $\C_N(P)\unlhd G$. Going over to $G/\C_N(P)$, we may assume that $\C_N(P)=1$. Let $P_1,\ldots,P_{p+1}\le P$ be maximal subgroups of $P$ such that $P=P_1\cup\ldots\cup P_{p+1}$. If $\C_N(P_i)=1$ for some $i$, then $P_i$ is redundant in $P_iN$ and $\nu_p(P_iN)=|N|=\nu_p(G)$. Arguing by induction on $|G|$, we can assume that $N_i:=\C_N(P_i)>1$ for $i=1,\ldots,p+1$. 
Using the Fitting decomposition as in the proof of \autoref{main}\eqref{c}, we obtain $N_1\times\ldots\times N_{p+1}\le N$. Since $P$ acts non-trivially on each $N_i$, it is clear that $|N_i|\ge q$. In total, $|N|\ge q^{p+1}$. 
\end{proof}

We remark that $G:=\PSL(2,11)$ has a redundant Sylow $2$-subgroup by \cite[Theorem~D]{MMM} and $\nu_2(G)=55$ is a product of only two primes. This indicates that \autoref{psolv} may not hold for arbitrary groups. 

For $x\in P\in\Syl_p(G)$ let $\lambda_G(x)=\lambda_G(\langle x\rangle)$. Gheri~\cite{Gheri} has introduced the following condition:
\begin{equation}\label{gheri}
\nu_p(G)^{|P|/p}\ge\prod_{x\in P}\lambda_G(x).
\end{equation}
He has shown in \cite[Theorem~B]{Gheri} that \eqref{gheri} holds for all finite groups if and only it holds for all almost simple groups. No counterexamples are known to exist.
This yields a conjectural bound for $\nu_p(G)$. 

\begin{Thm}\label{conj}
Suppose that $G$ has a redundant Sylow $p$-subgroup of order $p^n$. If $G$ satisfies \eqref{gheri}, then $\nu_p(G)\ge (p+1)^{(p^n-1)/(p^{n-1}-1)}>(p+1)^p$.
\end{Thm}
\begin{proof}
Let $x\in P\in\Syl_p(G)$. Since $P$ is redundant, there exists a Sylow $p$-subgroup $Q\ne P$ such that $x\in P\cap Q$. As in the proof of \autoref{lem}, we may choose $g\in\N_P(P\cap Q)\setminus Q$ such that $Q^g,Q^{g^2},\ldots,Q^{g^p}$ are distinct Sylow $p$-subgroups containing $x$. Hence, $\lambda_G(x)\ge p+1$. Moreover, $\lambda_G(1)=\nu_p(G)$. Now \eqref{gheri} implies
\[\nu_p(G)^{p^{n-1}}\ge\lambda_G(1)\prod_{x\in P\setminus\{1\}}\lambda_G(x)\ge\nu_p(G)(p+1)^{p^n-1}.\]
Since $P$ is non-cyclic, $n\ge 2$ and $\frac{p^n-1}{p^{n-1}-1}>p$. 
\end{proof}

If $n=2$ in \autoref{conj}, then $\nu_p(G)\ge (p+1)^{p+1}$. This coincides with \autoref{psolv}, whenever, $p$ is a Mersenne prime or $p=2$. The proof of \cite[Theorem~B]{Gheri} reduces \eqref{gheri} to an almost simple group $S$ such that $\nu_p(S)\le\nu_p(G)$. Then $S$ is a primitive permutation group of degree $\le\nu_p(S)$. If $\nu_p(G)$ is small, say $\nu_p(G)<2^{12}$, we can check \eqref{gheri} by running through the library of primitive groups in GAP~\cite{GAPnew}. 
We did not find examples among non-solvable groups improving the values in \autoref{refine}. 

Finally, we answer \cite[Question~8.8]{MMM}.

\begin{Thm}\label{Q88}
For every $n\in\NN$ there exists a constant $\delta_n<1$ with the following property: For every set of Sylow $p$-subgroups $P_1,\ldots,P_n$ of a finite group $G$ we have $G_p=P_1\cup\ldots\cup P_n$ or
\[|P_1\cup\ldots\cup P_n|<\delta_n|G_p|.\]
\end{Thm}
\begin{proof}
We assume that $G_p\ne P_1\cup\ldots\cup P_n$ and argue by induction on $n$. Let $P\in\Syl_p(G)\setminus\{P_1,\ldots,P_n\}$. A well-known theorem of Frobenius asserts that $|G_p|=a|P|$ for some integer $a\ge 2$ (see e.\,g. \cite{IsaacsRobinson}). If $n=1$, then the claim holds with $\delta_1=\frac{1}{2}$. Now let $n\ge 2$ and assume that $\delta_{n-1}$ is already given. Let $\rho_n$ be the smallest positive integer such that $\delta_{n-1}+\frac{1}{\rho_n}<1$. If $a\ge \rho_n$, then induction yields
\[|P_1\cup\ldots\cup P_n|\le|P_1\cup\ldots\cup P_{n-1}|+|P|\le \delta_{n-1}|G_p|+\frac{1}{a}|G_p|\le\Bigl(\delta_{n-1}+\frac{1}{\rho_n}\Bigr)|G_p|.\]
Now suppose that $a\le\rho_n$. We may assume that $P\nsubseteq P_1\cup\ldots\cup P_n$. Hence, by \cite[Theorem~1]{SambaleUnion}, there exists a constant $\gamma_n<1$ such that 
\begin{align*}
|P_1\cup\ldots\cup P_n|&\le |G_p\setminus P|+|(P\cap P_1)\cup\ldots\cup(P\cap P_n)|\le\frac{a-1}{a}|G_p|+\gamma_n|P|\\
&=\Bigl(1-\frac{1-\gamma_n}{a}\Bigr)|G_p|\le\Bigl(1-\frac{1-\gamma_n}{\rho_n}\Bigr)|G_p|.
\end{align*}
Finally, the claim holds with 
\[\delta_n:=\max\Bigl\{\delta_{n-1}+\frac{1}{\rho_n},\ 1-\frac{1-\gamma_n}{\rho_n}\Bigr\}.\qedhere\]
\end{proof}

\end{document}